\documentclass[11pt, reqno]{amsart}
\usepackage{amsmath, amsthm, amscd, amsfonts, amssymb, graphicx, color}
\usepackage[bookmarksnumbered, colorlinks, plainpages]{hyperref}

\input{mathrsfs.sty}

\newtheorem{theorem}{Theorem}[section]
\newtheorem{lemma}[theorem]{Lemma}

\newtheorem{corollary}[theorem]{Corollary}
\theoremstyle{definition}

\theoremstyle{remark}
\newtheorem{remark}[theorem]{Remark}
\numberwithin{equation}{section}
\begin{document}
	\title[Further norm and numerical radius inequalities ]{Further norm and numerical radius inequalities for sum of
Hilbert space operators}

 \author[D. Afraz, R. Lashkaripour, M. Bakherad ]{D. Afraz, R. Lashkaripour and M. Bakherad }

  \address{ Department of Mathematics, Faculty of Mathematics, University of Sistan and Baluchestan, Zahedan, Iran.}

\email{mojtaba.bakherad@yahoo.com}	
\email{davood.afraz@pgs.usb.ac.ir}
\email {lashkari@hamoon.usb.ac.ir}
\email{mojtaba.bakherad@yahoo.com; bakherad@member.ams.org}	
	
	\date{\today}
	
	
	\begin{abstract}
	Let ${\mathbb B}(\mathscr H)$ denotes the set of all bounded linear operators on a complex Hilbert space ${\mathscr H}$. In this paper, we present
some norm inequalities for sums of operators which are a generalization some recent results. Among other inequalities, it is shown that if $S, T\in {\mathbb B}({\mathscr H})$ are normal operators, then
\begin{eqnarray*}
&&\left\Vert S+T\right\Vert\\&\leq& \frac{1}{2}(\left\Vert S\right\Vert+\left\Vert T\right\Vert)+\frac{1}{2}\min_{t>0}\sqrt{ (\left\Vert S \right\Vert-\left\Vert T\right\Vert)^2+
\left\Vert  \frac{1}{t} f_1(\vert S \vert)g_1(\vert T\vert)+tf_2(\vert S \vert)g_2(\vert T\vert) \right\Vert^2},
\end{eqnarray*}	
where  $f_1,f_2,g_1,g_2$ are  non-negative continuous functions on $[0,\infty )$, in which $f_1(x)f_2(x)=x$ and $g_1(x)g_2(x)=x\,\,(x\geq 0)$. Moreover, several inequalities for the numerical radius are shown.
	\end{abstract}

\maketitle

\section{Introduction}

Let ${\mathbb B}(\mathscr H)$ denotes the $C^{*}$-algebra of all bounded linear operators on a complex Hilbert space ${\mathscr H}$ with an inner product $\langle \cdot,\cdot\rangle$ and the corresponding norm $ \Vert \cdot \Vert $. In the case when ${\rm dim}{\mathscr H}=n$, we identify ${\mathbb B}({\mathscr H})$ with the matrix algebra $\mathbb{M}_n$ of all $n\times n$ matrices with entries in
the complex field. For $ S\in {\mathbb B}(\mathscr H) $, let $ S = \Re(S) + i\Im(S) $
be the Cartesian decomposition of $S$, where the Hermitian matrices $ \Re(S)=\frac{S+S^{*}}{2} $ and $ \Im(S)=\frac{S-S^{*}}{2i} $ are called the real and imaginary parts of $S$, respectively.
 The numerical radius of $S\in {\mathbb B}({\mathscr H})$ is defined by
\begin{align*}
w(S):=\sup\{| \langle Sx, x\rangle| : x\in {\mathscr H}, \Vert x \Vert=1\}.
\end{align*}
It is well known that $w(\,\cdot\,)$ defines a norm on ${\mathbb B}({\mathscr H})$, which is equivalent to the usual operator norm $\Vert \cdot \Vert$. In fact, for any $S\in {\mathbb B}({\mathscr H})$,
$
\frac{1}{2}\Vert S \Vert\leq w(S) \leq\Vert S \Vert$;
 see. Let $r (\cdot)$ denote the spectral radius. It is well known that for every operator
$S\in{\mathbb B}({\mathscr H})$, we have $r(S)\leq w(S)$.
In, the authors showed that
$
  w\left( S \right) = \sup_{\theta\in\mathbb{R}}\Vert\Re({{\rm{e}}^{i\theta } S})\Vert=\sup_{\alpha^2+\beta^2=1}\Vert \alpha\Re S+\beta \Im S\Vert,
$
which is equali to the above definition.
For more facts about  the numerical radius see  and references therein.
 Let $S, T,X,Y\in {\mathbb B}({\mathscr H})$. The operator matrices  $\left[\begin{array}{cc} S&0\\ 0&Y \end{array}\right]$  and $\left[\begin{array}{cc}  0&T\\ X&0 \end{array}\right]$ are called the diagonal and off-diagonal parts of the operator matrix $\left[\begin{array}{cc} S&T\\ X&Y \end{array}\right]$, respectively.\\
In , it has been shown that if $S$ is an operator in ${\mathbb B}({\mathscr H})$, then
\begin{eqnarray}\label{abcd}
w(S)\leq \frac{1}{2}\left(\|S\|+\|S^2\|^\frac{1}{2}\right).
\end{eqnarray}%
Several refinements and generalizations of inequality  have been given; see  Yamazaki  showed that for  $S\in{\mathbb B}({\mathscr H})$ and  $t\in[0,1]$ we have
\begin{eqnarray}\label{yam-sheb}
w(S)\leq \frac{1}{2}\left(\|S\|+w(\tilde{S}_t)\right).
\end{eqnarray}%
where $S = U |S|$ is the polar decomposition of $S$ and $\tilde{S}_t =|S|^{t} U |S|^{1-t}$. Horn et al.  proved that
\begin{eqnarray}\label{horn}
\Vert S+T\Vert\leq \Vert \vert S\vert + \vert T\vert\Vert,
\end{eqnarray}%
where $S, T\in{\mathbb B}({\mathscr H})$ are normal. Davidson and Power  proved that if $S$ and $T$ are positive operators in ${\mathbb B}({\mathscr H})$, then
\begin{eqnarray}\label{abcd12}
\Vert S+T\Vert\leq \max\{\Vert S\Vert,\Vert T\Vert \}+\Vert ST\Vert^\frac{1}{2}.
\end{eqnarray}%
Inequality  has been generalized in . In , the author extended this inequality to the form
\begin{eqnarray}\label{abcd123}
\Vert S+T\Vert\leq \max\{\Vert S\Vert,\Vert T\Vert \}+\frac{1}{2}\left(\left\Vert|S|^t|T|^{1-t}\right\Vert+ \left\Vert|S^*|^{1-t}|T^*|^t\right\Vert\right),
\end{eqnarray}%
in which $S,T\in{\mathbb B}({\mathscr H})$ and $t\in[0,1]$.
In , the authors showed that a generalization of inequality  as follows:
\begin{eqnarray*}
 \Vert S+T\Vert\leq\max\{\Vert S\Vert,\Vert T\Vert\}+ \frac{1}{2}\Big(\big\Vert f(|S|)g(|T|)\big\Vert+\big\Vert f(|S^\ast|)g(|T^\ast|)\big\Vert\Big),
\end{eqnarray*}
in which  $S,T\in {\mathbb B}({\mathscr H})$  and $f,g$ be two non-negative non-decreasing continuous functions on $[0,\infty)$ such that $f(x)g(x)=x\,\,(x\geq0)$.
Recently, Shi et al.  proved the following inequality
\begin{eqnarray}\label{shi}
\left\Vert S+T\right\Vert\leq \frac{1}{2}(\left\Vert S\right\Vert+\left\Vert T\right\Vert)+\frac{1}{2}\sqrt{ (\left\Vert S \right\Vert-\left\Vert T\right\Vert)^2+
\left\Vert  \frac{1}{t} \vert S \vert^r\vert T\vert^s+t\vert S \vert^{1-r} \vert T\vert^{1-s}  \right\Vert^2}
\end{eqnarray}
for normal operators $S,T\in {\mathbb B}({\mathscr H})$, $r,s\in[0,1]$ and $t>0$.

In this paper, we give several further norm inequalities for sums of bounded linear operators. These inequalities refine and generalize inequalities.  Moreover, as another application, we show a new numerical radius inequality which is a generalization.


\section{main results}
To present our results, we need the following lemmas.
\begin{lemma}
Let $S, T, X, Y\in {\mathbb B}({\mathscr H})$. Then
\begin{eqnarray*}
\left\Vert\left[
\begin{array}{cc}
S & T \\
X & Y%
\end{array}%
\right]\right\Vert\leq
\left\Vert\left[
\begin{array}{cc}
\Vert S\Vert & \Vert T \Vert\\
\Vert X\Vert & \Vert Y\Vert%
\end{array}%
\right]\right\Vert.
\end{eqnarray*}
\end{lemma}
\begin{lemma}
If $S, T\in {\mathbb B}({\mathscr H})$ in which $ST$ is selfadjoint, then
\begin{eqnarray*}
\Vert ST\Vert \leq \Vert \Re(TS)\Vert.
\end{eqnarray*}
\end{lemma}
In the first result, a generalization of the inequality  is obtained.
\begin{theorem}\label{the1}
Let $S, T\in {\mathbb B}({\mathscr H})$ be normal and $f_1,f_2,g_1,g_2$ be  non-negative
continuous functions on $[0,\infty )$, in which $f_1(x)f_2(x)=x$ and $g_1(x)g_2(x)=x\,\,(x\geq 0)$. Then
\begin{eqnarray*}
&&\left\Vert S+T\right\Vert\\&\leq& \frac{1}{2}(\left\Vert S\right\Vert+\left\Vert T\right\Vert)+\frac{1}{2}\min_{t>0}\sqrt{ (\left\Vert S \right\Vert-\left\Vert T\right\Vert)^2+
\left\Vert  \frac{1}{t} f_1(\vert S \vert)g_1(\vert T\vert)+tf_2(\vert S \vert)g_2(\vert T\vert) \right\Vert^2}.
\end{eqnarray*}
\end{theorem}
\begin{proof}
Let $S$ and $T$ be positive operators. Then for all $t>0$ we have
\begin{eqnarray*}
&&\left\Vert S+T\right\Vert\\&=&
\left\Vert\left[
\begin{array}{cc}
S+T & 0 \\
0 & 0%
\end{array}%
\right]\right\Vert\\&=&
\left\Vert\left[
\begin{array}{cc}
t f_2(S) & g_1(T) \\
0 & 0%
\end{array}%
\right]
\left[
\begin{array}{cc}
\frac{1}{t} f_1(S) & 0 \\
g_2(T) & 0%
\end{array}%
\right]\right\Vert\\&\leq&
\left\Vert\Re\left(
\left[
\begin{array}{cc}
\frac{1}{t} f_1(S) & 0 \\
g_2(T) & 0%
\end{array}%
\right]\left[
\begin{array}{cc}
t f_2(S) & g_1(T) \\
0 & 0%
\end{array}%
\right]\right)\right\Vert\\&&\qquad\qquad\qquad\qquad\qquad\qquad(\textrm{by Lemma })\\&=&
\left\Vert\left[
\begin{array}{cc}
S & \frac{1}{2}(\frac{1}{t}f_1(S)g_1(T)+tf_2(S)g_2(T)) \\
\frac{1}{2}(tf_2(S)g_2(T)+\frac{1}{t}f_1(S)g_1(T)) & T%
\end{array}%
\right]
\right\Vert\\&\leq&
\left\Vert\left[
\begin{array}{cc}
\Vert S\Vert & \frac{1}{2}\Vert\frac{1}{t}f_1(S)g_1(T)+tf_2(S)g_2(T)\Vert \\
\frac{1}{2}\Vert tf_2(S)g_2(T)+\frac{1}{t}f_1(S)g_1(T)\Vert & \Vert T\Vert%
\end{array}%
\right]\right\Vert
\\&&\qquad\qquad\qquad\qquad\qquad\qquad(\textrm{by Lemma })\\&=&
r\left(\left[
\begin{array}{cc}
\Vert S\Vert & \frac{1}{2}\Vert\frac{1}{t}f_1(S)g_1(T)+tf_2(S)g_2(T)\Vert \\
\frac{1}{2}\Vert\frac{1}{t}f_2(S)g_2(T)+tf_1(S)g_1(T)\Vert & \Vert T\Vert%
\end{array}%
\right]
\right)\\&&\qquad\qquad\qquad\qquad\qquad\qquad(\textrm{since is Hermitian matrix})\\&=&
\frac{1}{2}(\left\Vert S\right\Vert+\left\Vert T\right\Vert)+\frac{1}{2}\sqrt{ (\left\Vert S \right\Vert-\left\Vert T\right\Vert)^2+
\left\Vert  \frac{1}{t} f_1(\vert S \vert)g_1(\vert T\vert)+tf_2(\vert S \vert)g_2(\vert T\vert) \right\Vert^2},
\end{eqnarray*}
whence
\begin{eqnarray*}
&&\left\Vert S+T\right\Vert\\&\leq&
\frac{1}{2}(\left\Vert S\right\Vert+\left\Vert T\right\Vert)+\frac{1}{2}\min_{t>0}\sqrt{ (\left\Vert S \right\Vert-\left\Vert T\right\Vert)^2+
\left\Vert  \frac{1}{t} f_1(\vert S \vert)g_1(\vert T\vert)+tf_2(\vert S \vert)g_2(\vert T\vert) \right\Vert^2}
\end{eqnarray*}
for positive operators $S$ and $T$. Now, if $S$ and $T$ are normal operators, then by using inequality we have
\begin{eqnarray*}
&&\left\Vert S+T\right\Vert\\&\leq&
\left\Vert \vert S\vert+\vert T\vert\right\Vert\qquad(\textrm{by inequality })
\\&\leq&
\frac{1}{2}(\left\Vert \vert S\vert\right\Vert+\left\Vert \vert T\vert\right\Vert)+\frac{1}{2}\min_{t>0}\sqrt{ (\left\Vert \vert S\vert \right\Vert-\left\Vert \vert T\vert\right\Vert)^2+
\left\Vert  \frac{1}{t} f_1(\vert S \vert)g_1(\vert T\vert)+tf_2(\vert S \vert)g_2(\vert T\vert) \right\Vert}
\\&=&
\frac{1}{2}(\left\Vert S\right\Vert+\left\Vert T\right\Vert)+\frac{1}{2}\min_{t>0}\sqrt{ (\left\Vert S \right\Vert-\left\Vert T\right\Vert)^2+
\left\Vert  \frac{1}{t} f_1(\vert S \vert)g_1(\vert T\vert)+tf_2(\vert S \vert)g_2(\vert T\vert) \right\Vert}.
\end{eqnarray*}
\end{proof}
As a consequence of Theorem  we have the following result.
\begin{corollary}\label{cor1}
Let $S, T\in {\mathbb B}({\mathscr H})$ be normal, $f_1,f_2,g_1,g_2$ be  non-negative
continuous functions on $[0,\infty )$ such that $f_1(x)f_2(x)=x$ and $g_1(x)g_2(x)=x\,\,(x\geq 0)$. Then
\begin{eqnarray*}
\left\Vert S+T\right\Vert&\leq& \frac{1}{2}(\left\Vert S\right\Vert+\left\Vert T\right\Vert)+\frac{1}{2}\sqrt{ (\left\Vert S \right\Vert-\left\Vert T\right\Vert)^2+
4\left\Vert  f_1(\vert S \vert)g_1(\vert T\vert)\right\Vert \left\Vert f_2(\vert S \vert)g_2(\vert T\vert) \right\Vert}.
\end{eqnarray*}
\end{corollary}
\begin{proof}
If $S$ and $T$ are normal operators, then
\begin{eqnarray*}
&&\left\Vert S+T\right\Vert\\&\leq& \frac{1}{2}(\left\Vert S\right\Vert+\left\Vert T\right\Vert)+\frac{1}{2}\min_{t>0}\sqrt{ (\left\Vert S \right\Vert-\left\Vert T\right\Vert)^2+
\left\Vert  \frac{1}{t} f_1(\vert S \vert)g_1(\vert T\vert)+tf_2(\vert S \vert)g_2(\vert T\vert) \right\Vert^2}
\\&\leq&
\frac{1}{2}(\left\Vert S\right\Vert+\left\Vert T\right\Vert)+\frac{1}{2}\min_{t>0}\sqrt{ (\left\Vert S \right\Vert-\left\Vert T\right\Vert)^2+
\left(\frac{1}{t}\left\Vert   f_1(\vert S \vert)g_1(\vert T\vert)\right\Vert+t\left\Vert f_2(\vert S \vert)g_2(\vert T\vert) \right\Vert\right)^2}
\\&=&
\frac{1}{2}(\left\Vert S\right\Vert+\left\Vert T\right\Vert)+\frac{1}{2}\sqrt{ (\left\Vert S \right\Vert-\left\Vert T\right\Vert)^2+
4\left\Vert  f_1(\vert S \vert)g_1(\vert T\vert)\right\Vert \left\Vert f_2(\vert S \vert)g_2(\vert T\vert) \right\Vert}.
\end{eqnarray*}
The last equation follows from $\min_{t>0} h(t)=\min_{t>0} (\frac{1}{t}a+tb)=2\sqrt{ab}\,\,(a,b\geq0)$.
\end{proof}
For invertible normal operators we get the next result.
\begin{corollary}\label{cor2}
Let $S, T\in {\mathbb B}({\mathscr H})$ be invertible normal and $f, g$ be two positive
continuous functions on $[0,\infty )$. Then
\begin{eqnarray*}
&&\left\Vert S+T\right\Vert\\&\leq& \frac{1}{2}(\left\Vert S\right\Vert+\left\Vert T\right\Vert)+\frac{1}{2}\min_{t>0}\sqrt{ (\left\Vert S \right\Vert-\left\Vert T\right\Vert)^2+
\left\Vert  \frac{1}{t} f(\vert S \vert)g(\vert T\vert)+t\vert S \vert (f(\vert S \vert))^{-1}\vert T\vert (g(\vert T\vert))^{-1} \right\Vert^2}.
\end{eqnarray*}
In particular,
\begin{eqnarray*}
\left\Vert S+T\right\Vert\leq \frac{1}{2}(\left\Vert S\right\Vert+\left\Vert T\right\Vert)+\frac{1}{2}\min_{t>0}\sqrt{ (\left\Vert S \right\Vert-\left\Vert T\right\Vert)^2+
\left\Vert  \frac{1}{t} \vert S \vert^r\vert T\vert^s+t\vert S \vert^{1-r} \vert T\vert^{1-s}  \right\Vert^2},
\end{eqnarray*}
where $r,s\in\mathbb{R}$.
\end{corollary}
\begin{proof}
For the first inequality, it is enough to put $f_1(t)=f(t)$, $f_2(t)=\frac{t}{f(t)}$, $g_1(t)=g(t)$ and  $g_2(t)=\frac{t}{g(t)}\,\,(t>0)$ in Theorem . In particular, if $f(t)=t^r$ and $g(t)=t^s\,\,(r,s\in\mathbb{R})$ we get the second inequality.
\end{proof}

\begin{corollary}\label{cor2}
Let $S, T\in {\mathbb B}({\mathscr H})$ be invertible normal and $f, g$ be two positive
continuous functions on $[0,\infty )$. Then
\begin{eqnarray*}
\left\Vert S+T\right\Vert\leq \max\{\left\Vert S\right\Vert+\left\Vert T\right\Vert\}+
\left\Vert   f(\vert S \vert)g(\vert T\vert)\right\Vert^\frac{1}{2}\left\Vert\vert S \vert (f(\vert S \vert))^{-1}\vert T\vert (g(\vert T\vert))^{-1} \right\Vert^\frac{1}{2}.
\end{eqnarray*}
In particular,
\begin{eqnarray}\label{kit1}
\left\Vert S+T\right\Vert\leq \max\{\left\Vert S\right\Vert,\left\Vert T\right\Vert\}+
\left\Vert   \vert S \vert^r\vert T\vert^{1-s}\right\Vert^\frac{1}{2}\left\Vert\vert S \vert^{1-r} \vert T\vert^{s} \right\Vert^\frac{1}{2},
\end{eqnarray}
where $r,s\in\mathbb{R}$.
\end{corollary}
\begin{proof}
If $S, T\in {\mathbb B}({\mathscr H})$ is invertible normal and $f, g$ are two positive
continuous functions on $[0,\infty )$, then by using Theorem  we have
\begin{eqnarray*}
&&\left\Vert S+T\right\Vert
\\&\leq&
\frac{1}{2}(\left\Vert S\right\Vert+\left\Vert T\right\Vert)+\frac{1}{2}\sqrt{ (\left\Vert S \right\Vert-\left\Vert T\right\Vert)^2+
\left\Vert  \frac{1}{t} f(\vert S \vert)g(\vert T\vert)+t\vert S \vert (f(\vert S \vert))^{-1}\vert T\vert (g(\vert T\vert))^{-1} \right\Vert^2}
\\&\leq&
\frac{1}{2}(\left\Vert S\right\Vert+\left\Vert T\right\Vert)+\frac{1}{2} \vert\left\Vert S \right\Vert-\left\Vert T\right\Vert\vert+
\frac{1}{2}\left\Vert  \frac{1}{t} f(\vert S \vert)g(\vert T\vert)+t\vert S \vert (f(\vert S \vert))^{-1}\vert T\vert (g(\vert T\vert))^{-1} \right\Vert
\\&&\qquad\qquad\qquad\qquad\qquad\qquad(\textrm{by the inequality}\,\,\sqrt{x^2+y^2}\leq\sqrt{x^2}+\sqrt{y^2})
\\&=&
\max\{\Vert S\Vert, \Vert T\Vert\}+
\left\Vert  \frac{1}{t} f(\vert S \vert)g(\vert T\vert)+t\vert S \vert (f(\vert S \vert))^{-1}\vert T\vert (g(\vert T\vert))^{-1} \right\Vert
\\&\leq&
\max\{\Vert S\Vert, \Vert T\Vert\}+
\frac{1}{2}\left(\frac{1}{t}\left\Vert   f(\vert S \vert)g(\vert T\vert)\right\Vert +t\left\Vert  \vert S \vert (f(\vert S \vert))^{-1}\vert T\vert (g(\vert T\vert))^{-1} \right\Vert\right).
\end{eqnarray*}
Now, we take the minimum on $t\,\,(t>0)$ and we get
  \begin{eqnarray*}
&&\left\Vert S+T\right\Vert
\\&\leq&
\max\{\Vert S\Vert, \Vert T\Vert\}+
\frac{1}{2}\min_{t>0}\left(\frac{1}{t}\left\Vert   f(\vert S \vert)g(\vert T\vert)\right\Vert +t\left\Vert  \vert S \vert (f(\vert S \vert))^{-1}\vert T\vert (g(\vert T\vert))^{-1} \right\Vert\right)
\\&=&
\max\{\left\Vert S\right\Vert,\left\Vert T\right\Vert\}+
\left\Vert   f(\vert S \vert)g(\vert T\vert)\right\Vert^\frac{1}{2}\left\Vert\vert S \vert (f(\vert S \vert))^{-1}\vert T\vert (g(\vert T\vert))^{-1} \right\Vert^\frac{1}{2}.
\end{eqnarray*}
For the second inequality, put $f(t)=t^r$ and $g(t)=t^{1-s}\,\,(r,s\in\mathbb{R})$ in the first inequality.
\end{proof}
\begin{remark}
If $S, T\in {\mathbb B}({\mathscr H})$ are normal operator, then inequality is a refinement and generalization  of inequality  for normal operators. In fact, in this case
\begin{eqnarray*}
\left\Vert S+T\right\Vert
&\leq&
\max\{\left\Vert S\right\Vert,\left\Vert T\right\Vert\}+
\left\Vert   \vert S \vert^r\vert T\vert^{1-s}\right\Vert^\frac{1}{2}\left\Vert\vert S \vert^{1-r} \vert T\vert^{s} \right\Vert^\frac{1}{2}
\\&\leq&
\max\{\left\Vert S\right\Vert,\left\Vert T\right\Vert\}+
\frac{1}{2}\left(\left\Vert   \vert S \vert^r\vert T\vert^{1-s}\right\Vert+\left\Vert\vert S \vert^{1-r} \vert T\vert^{s} \right\Vert\right)
\\&&\qquad\qquad\qquad\qquad(\textrm{by the arithmetic-geometric  inequality}),
\end{eqnarray*}
where $r,s\in[0,1]$.
\end{remark}

In the next result, we show a generalization of the inequality   for arbitrary operators in $S, T\in{\mathbb B}({\mathscr H})$.
\begin{theorem}\label{the2}
Let $S, T\in {\mathbb B}({\mathscr H})$ and  $f_1,f_2,g_1,g_2$ be  non-negative
continuous functions on $[0,\infty )$, in which $f_1(x)f_2(x)=x$ and $g_1(x)g_2(x)=x\,\,(x\geq 0)$. Then
\begin{eqnarray*}
\left\Vert S+T\right\Vert\leq \frac{1}{2}(\left\Vert S\right\Vert+\left\Vert T\right\Vert)+\frac{1}{2}\sqrt{ (\left\Vert S \right\Vert-\left\Vert T\right\Vert)^2+
\max\{\alpha,\beta\}},
\end{eqnarray*}
where $\alpha=\left\Vert  \frac{1}{t} f_1(\vert S^* \vert)g_1(\vert T^*\vert)+tf_2(\vert S^* \vert)g_2(\vert T^*\vert) \right\Vert^2$ and $\beta=\left\Vert  \frac{1}{t} f_1(\vert S \vert)g_1(\vert T\vert)+tf_2(\vert S \vert)g_2(\vert T\vert) \right\Vert^2$.
\end{theorem}
\begin{proof}
Let $\tilde{S}=\left[
\begin{array}{cc}
0 & S \\
S^* & 0%
\end{array}%
\right]$
and
$\tilde{T}=\left[
\begin{array}{cc}
0 & T \\
T^* & 0%
\end{array}%
\right]$. Then $\tilde{S}$ and $\tilde{T}$ are normal operators and
\[
\Vert \tilde{S}+\tilde{T}\Vert=\left\Vert \left[
\begin{array}{cc}
0 & S+T \\
S^*+S^* & 0%
\end{array}%
\right]\right\Vert=\Vert S+T\Vert.
\]
Hence, applying Theorem  we get
\begin{eqnarray}\label{prod}
&&\left\Vert S+T\right\Vert
=
\Vert \tilde{S}+\tilde{T}\Vert\nonumber
\\&\leq&
\frac{1}{2}(\Vert \tilde{S}\Vert+\Vert \tilde{T}\Vert)+\frac{1}{2}\sqrt{ (\Vert \tilde{S} \Vert-\Vert \tilde{T}\Vert)^2+
\left\Vert  \frac{1}{t} f_1(\vert \tilde{S} \vert)g_1(\vert \tilde{T}\vert)+tf_2(\vert \tilde{S} \vert)g_2(\vert \tilde{T}\vert) \right\Vert^2}.
\end{eqnarray}
Moreover, it follows from	$\Vert \tilde{S}\Vert=\Vert S\Vert$, $\Vert \tilde{T}\Vert=\Vert T\Vert$ and
\[
f(\vert \tilde{S} \vert)=\left[
\begin{array}{cc}
f(\vert S^*\vert)&0 \\
0 & f(\vert S\vert)%
\end{array}%
\right],\quad f(\vert \tilde{T} \vert)=\left[
\begin{array}{cc}
f(\vert T^*\vert)&0 \\
0 & f(\vert T\vert)%
\end{array}%
\right]
\]
that we get
\begin{eqnarray*}
&&\left\Vert  \frac{1}{t} f_1(\vert \tilde{S} \vert)g_1(\vert \tilde{T}\vert)+tf_2(\vert \tilde{S} \vert)g_2(\vert \tilde{T}\vert) \right\Vert
\\&=&
\left\Vert\left[
\begin{array}{cc}
\frac{1}{t} f_1(\vert {S^*} \vert)g_1(\vert {T^*}\vert)+tf_2(\vert {S^*} \vert)g_2(\vert {T^*}\vert)&0 \\
0 & \frac{1}{t} f_1(\vert {S} \vert)g_1(\vert {T}\vert)+tf_2(\vert {S} \vert)g_2(\vert {T}\vert)%
\end{array}%
\right]\right\Vert
\\&=&
\max\left\{\left\Vert\frac{1}{t} f_1(\vert {S^*} \vert)g_1(\vert {T^*}\vert)+tf_2(\vert {S^*} \vert)g_2(\vert {T^*}\vert)\right\Vert,\left\Vert\frac{1}{t} f_1(\vert {S} \vert)g_1(\vert {T}\vert)+tf_2(\vert {S} \vert)g_2(\vert {T}\vert)\right\Vert\right\}.
\end{eqnarray*}
Using the recent equality and inequality  we reach the desired result.
\end{proof}
Applying Theorem  and a same argument in the proof of Corollary  we get the next result.
\begin{corollary}
Let $S, T\in {\mathbb B}({\mathscr H})$. Then
\begin{eqnarray*}
&&\left\Vert S+T\right\Vert\\&\leq&
\max\left\{\left\Vert S\right\Vert,\left\Vert T\right\Vert\right\}+
\max \left\{ \Vert  \vert S^* \vert^r\vert T^*\vert^{1-s}\Vert^{\frac{1}{2}}\Vert\vert S^* \vert^{1-r}\vert T^*\vert^s \Vert^{\frac{1}{2}}, \Vert\vert S \vert^r\vert T\vert^{1-s}\Vert^{\frac{1}{2}} \Vert \vert S \vert^{1-r}\vert T\vert^s\Vert^{\frac{1}{2}}  \right\},
\end{eqnarray*}
where $r,s\in[0,1]$. In particular,
\begin{eqnarray}\label{num}
\left\Vert S+T\right\Vert\leq
\max\left\{\left\Vert S\right\Vert,\left\Vert T\right\Vert\right\}+
\max \left\{ \Vert  \vert S^* \vert^{\frac{1}{2}}\vert T^*\vert^{\frac{1}{2}}\Vert, \Vert\vert S \vert^{\frac{1}{2}}\vert T\vert^{\frac{1}{2}}\Vert   \right\}.
\end{eqnarray}
\end{corollary}
\section{Some results for the numerical radius}
In this section, as application of norm inequalities for sums of operators, we present some inequalities for the numerical radius.
\begin{theorem}\label{the3}
Let $S, T\in {\mathbb B}({\mathscr H})$. Then
\begin{eqnarray*}
w(S-T)&\geq&\max\{2w(S),2w(T)\}
 -\max\left\{\left\Vert S\right\Vert,\left\Vert T\right\Vert\right\}
 \\&&-
\max \left\{ \Vert  \vert S^* \vert^r\vert T^*\vert^{1-s}\Vert^{\frac{1}{2}}\Vert\vert S^* \vert^{1-r}\vert T^*\vert^s \Vert^{\frac{1}{2}}, \Vert\vert S \vert^r\vert T\vert^{1-s}\Vert^{\frac{1}{2}} \Vert \vert S \vert^{1-r}\vert T\vert^s\Vert^{\frac{1}{2}} \right\},
\end{eqnarray*}
where $r,s\in[0,1]$.
\end{theorem}
\begin{proof}
If $S, T\in {\mathbb B}({\mathscr H})$, then
\begin{eqnarray*}
2\max\{w(S),w(T)\}&=&
w\left(\left[
\begin{array}{cc}
S&0 \\
0 & T%
\end{array}%
\right]\right)
\\&=&
w\left(\left[
\begin{array}{cc}
S+T&0 \\
0 & T+S%
\end{array}%
\right]+\left[
\begin{array}{cc}
S-T&0 \\
0 & T-S%
\end{array}%
\right]\right)
\\&\leq&
w\left(\left[
\begin{array}{cc}
S+T&0 \\
0 & T+S%
\end{array}%
\right]\right)+w\left(\left[
\begin{array}{cc}
S-T&0 \\
0 & T-S%
\end{array}%
\right]\right)
\\&=&
w(S+T)+w(S-T),
\end{eqnarray*}
whence
\begin{eqnarray*}
w(S-T)\geq2\max\{w(S),w(T)\}-w(S+T).
\end{eqnarray*}
It follows from $w(S+T)\leq\Vert S+T\Vert$ and Corollary  that
\begin{eqnarray*}
w(S-T)&\geq&\max\{2w(S),2w(T)\}-w(S+T)
\\&\geq&
\max\{2w(S),2w(T)\}-\Vert S+T\Vert
 \\&\geq&
 \max\{2w(S),2w(T)\}
 -\max\left\{\left\Vert S\right\Vert,\left\Vert T\right\Vert\right\}
 \\&&-
\max \left\{ \Vert  \vert S^* \vert^r\vert T^*\vert^{1-s}\Vert^{\frac{1}{2}}\Vert\vert S^* \vert^{1-r}\vert T^*\vert^s \Vert^{\frac{1}{2}}, \Vert\vert S \vert^r\vert T\vert^{1-s}\Vert^{\frac{1}{2}} \Vert \vert S \vert^{1-r}\vert T\vert^s\Vert^{\frac{1}{2}} \right\}
\\&&\qquad\qquad\qquad\qquad\qquad\qquad(\textrm{by Corollary }),
\end{eqnarray*}
where $r,s\in[0,1]$.
\end{proof}
\begin{remark}
If $S, T\in {\mathbb B}({\mathscr H})$ are normal operator, then $\vert S\vert=\vert S^*\vert$, $\vert T\vert=\vert T^*\vert$, $w(S)=\Vert S\Vert$ and $w(T)=\Vert T\Vert$. This concludes that Theorem  appear as
\begin{eqnarray*}
\Vert S-T\Vert\geq
 \max\left\{\left\Vert S\right\Vert,\left\Vert T\right\Vert\right\}
   -\Vert  \vert S \vert^r\vert T\vert^{1-s}\Vert^{\frac{1}{2}}\Vert\vert S \vert^{1-r}\vert T\vert^s \Vert^{\frac{1}{2}}.
\end{eqnarray*}
In particular, if $S$ and $T$ are positive, then for $r=s=\frac{1}{2}$ we have
\begin{eqnarray*}
\Vert S-T\Vert\geq
 \max\left\{\left\Vert S\right\Vert,\left\Vert T\right\Vert\right\}
 -  \Vert   S ^\frac{1}{2} T^\frac{1}{2}\Vert.
\end{eqnarray*}
\end{remark}
In the next result, we obtain an upper bound for the numerical radius.
\begin{theorem}\label{the3}
Let $S\in {\mathbb B}({\mathscr H})$ and  $f_1,f_2,g_1,g_2$ be  non-negative
continuous functions on $[0,\infty )$, in which $f_1(x)f_2(x)=x$ and $g_1(x)g_2(x)=x\,\,(x\geq 0)$. Then
\begin{eqnarray*}
w(S)\leq \frac{1}{2}\left\Vert S\right\Vert+\frac{1}{4}
\max\{\alpha,\beta\},
\end{eqnarray*}
in which $\alpha=\left\Vert   f_1(\vert S^* \vert)g_1(\vert S\vert)+f_2(\vert S^* \vert)g_2(\vert S\vert) \right\Vert$ and $\beta=\left\Vert   f_1(\vert S \vert)g_1(\vert S^*\vert)+f_2(\vert S \vert)g_2(\vert S^*\vert) \right\Vert$.
In particular,
\begin{eqnarray*}
w(S)\leq\frac{1}{2}\Vert S\Vert+\frac{1}{4}\left\Vert   f_1(\vert S^* \vert)f_2(\vert S\vert)+f_2(\vert S^* \vert)f_1(\vert S\vert) \right\Vert.
\end{eqnarray*}
\end{theorem}
\begin{proof}
Using the of $w(\cdot)$ and Theorem  for $t=1$ we have
\begin{eqnarray*}
w(S)&=& \sup_{\theta\in\mathbb{R}}\Vert\Re({{\rm{e}}^{i\theta } T})\Vert
\\&=& \frac{1}{2}\sup_{\theta\in\mathbb{R}}\Vert{{\rm{e}}^{i\theta } S}+{{\rm{e}}^{-i\theta } S^*}\Vert
\\&\leq&\frac{1}{4}(\left\Vert S\right\Vert+\left\Vert S^*\right\Vert)+\frac{1}{4}
\max\{\alpha,\beta\}\qquad(\textrm{by Theorem })
\\&=&\frac{1}{2}\left\Vert S\right\Vert+\frac{1}{4}
\max\{\alpha,\beta\},
\end{eqnarray*}
in which $\alpha=\left\Vert   f_1(\vert S^* \vert)g_1(\vert S\vert)+f_2(\vert S^* \vert)g_2(\vert S\vert) \right\Vert$ and $\beta=\left\Vert   f_1(\vert S \vert)g_1(\vert S^*\vert)+f_2(\vert S \vert)g_2(\vert S^*\vert) \right\Vert$. In the special case for $g_1=f_2$ and $g_2=f_1$   we get the second inequality.
\end{proof}
\begin{theorem}\label{the3}
Let $S\in {\mathbb B}({\mathscr H})$. Then
\begin{eqnarray*}
w(S)\leq \max\left\{\left\Vert \Re S\right\Vert,\left\Vert  \Im S\right\Vert\right\}+
\frac{1}{2} \Vert  \vert \Re S \vert^{\frac{1}{2}}\vert  \Im S\vert^{\frac{1}{2}}\Vert.
\end{eqnarray*}
\end{theorem}
\begin{proof}
Using inequality  and the definition $w(\cdot)$ we have
\begin{eqnarray*}
w(S)&=&\sup_{\alpha^2+\beta^2=1}\Vert \alpha\Re S+\beta \Im S\Vert
\\&\leq&
\sup_{\alpha^2+\beta^2=1}\left(\max\left\{\left\Vert \alpha\Re S\right\Vert,\left\Vert  \beta\Im S\right\Vert\right\}+
\max \left\{ \Vert  \vert \alpha\Re S \vert^{\frac{1}{2}}\vert  \beta\Im S\vert^{\frac{1}{2}}\Vert, \Vert\vert \alpha\Re S \vert^{\frac{1}{2}}\vert \beta \Im S\vert^{\frac{1}{2}}\Vert   \right\}\right)
\\&&\qquad\qquad\qquad\qquad\qquad(\textrm{by inequality })
\\&\leq&
\max\left\{\left\Vert \Re S\right\Vert,\left\Vert  \Im S\right\Vert\right\}+
\sup_{\alpha^2+\beta^2=1}\left(\sqrt{\vert \alpha\beta\vert} \Vert  \vert \Re S \vert^{\frac{1}{2}}\vert  \Im S\vert^{\frac{1}{2}}\Vert  \right)
\\&\leq&
\max\left\{\left\Vert \Re S\right\Vert,\left\Vert  \Im S\right\Vert\right\}+
\frac{\sqrt{2}}{2} \Vert  \vert \Re S \vert^{\frac{1}{2}}\vert  \Im S\vert^{\frac{1}{2}}\Vert.
\end{eqnarray*}
\end{proof}
\noindent  {Conflict of Interest:} The author declare that there is no conflict of interest.\\

\noindent  {Funding:}  The authors declare that there is no research grants from funding agencies, universities, or other parties.

\end{document}